\documentclass[draft,11pt,a4paper]{article}

\usepackage{amsfonts,amssymb}
\usepackage{amsthm,amsmath,amstext}

\newtheorem{theorem}{Theorem}
\newtheorem{lemma}[theorem]{Lemma}
\newtheorem{proposition}[theorem]{Proposition}

\newcommand{\G}[2]{G_{#1,#2}}
\newcommand{\Exp}{\,\mathbb{E}}
\renewcommand{\Pr}{\,\mathbb{P}}
\newcommand{\eps}{\varepsilon}

\usepackage{tikz}
\usetikzlibrary{backgrounds}
\usetikzlibrary{shapes}
\usepackage{caption}

\begin{document}
\title{A precise threshold for quasi-Ramsey numbers}
\author{
Ross J. Kang
\thanks{
Radboud University Nijmegen, Netherlands.
Email: {\tt ross.kang@gmail.com}.
This author was supported by a NWO Veni grant. This work was begun while this author was at Durham University, supported by EPSRC grant EP/G066604/1.
}
\and
J\'anos Pach
\thanks{
EPFL, Lausanne and R\'enyi Institute, Budapest. Supported
by Hungarian Science Foundation EuroGIGA Grant OTKA NN 102029, by Swiss National Science Foundation Grants 200020-144531 and 200021-137574.
Email: {\tt pach@cims.nyu.edu}.
}
\and
Viresh Patel
\thanks{
University of Amsterdam, Netherlands.
Email: {\tt V.S.Patel@uva.nl}.
This author was supported by the Queen Mary -- Warwick Strategic Alliance.
This work was begun while this author was at Durham University, supported by EPSRC grant EP/F064551/1.
}
\and
Guus Regts
\thanks{
University of Amsterdam, Netherlands.
Email: {\tt guusregts@gmail.com}.
The research leading to these results has received funding from the European Research Council
under the European Union's Seventh Framework Programme (FP7/2007-2013) / ERC grant agreement
n$\mbox{}^{\circ}$ 339109.
}
}
\date{\today}
\maketitle

\begin{abstract}
We consider a variation of Ramsey numbers introduced by Erd\H{o}s and Pach~\cite{ErPa83}, where instead of seeking complete or independent sets we only seek a {\em $t$-homogeneous set}, a vertex subset that induces a subgraph of minimum degree at least $t$ or the complement of such a graph.

For any $\nu > 0$ and positive integer $k$, we show that any graph $G$ or its complement contains as an induced subgraph some graph $H$ on $\ell \ge k$ vertices with minimum degree at least $\frac12(\ell-1) + \nu$ provided that $G$ has at least $k^{\Omega(\nu^2)}$ vertices.
We also show this to be best possible in a sense.
This may be viewed as correction to a result claimed in~\cite{ErPa83}.

For the above result, we permit $H$ to have order at least $k$.  In the harder problem where we insist that $H$ have exactly $k$ vertices, we do not obtain sharp results, although we show a way to translate results of one form of the problem to the other.
\end{abstract}

%\begin{keywords}
%Ramsey numbers, threshold, quasi-Ramsey 
%\end{keywords}

%%%%%%%%%%%%%%%%%%%%%%%%%%%%%%%%%%%%%%%%%%%%%%%%%%%%%%%%%%%%%%%%%%%%%%%%%

\section{Introduction}
\label{sec:intro}

Recall that the (diagonal, two-colour) Ramsey number is defined to be the smallest integer $R(k)$ for which any graph on $R(k)$ vertices is guaranteed to contain a homogeneous set of order $k$ --- that is, a set of $k$ vertices corresponding to either a complete or independent subgraph.
The development of asymptotic bounds for these numbers is an important and challenging area of mathematics with a history of more than eighty years.   
Since the work of Erd\H{o}s and Szekeres~\cite{ErSz35} and Erd\H{o}s~\cite{Erd47}, there has been no progress in improving bounds on the first-order term of $\ln R(k)$, so even seemingly small improvements in asymptotic bounds on $R(k)$ are of major importance~\cite{Con09}.

We consider a degree-based generalisation of $R(k)$ where, rather than seeking a clique or coclique of order at least $k$, we seek instead an induced subgraph of order at least $k$ with high minimum degree (clique-like graphs) or low maximum degree (coclique-like graphs). We call this the \emph{variable quasi-Ramsey} problem. By gradually relaxing the
degree requirement, we get a spectrum of Ramsey-type problems where we see a sharp change at a certain point. 
 Erd\H{o}s and Pach~\cite{ErPa83} introduced such problems and obtained some interesting results summarised below. 

\subsection{The variable quasi-Ramsey problem}

For a graph $G=(V,E)$, 
we write $\overline{G}$ for the complement of $G$. As a starting point, Erd\H{o}s and Pach observed the following.

\begin{proposition} [\cite{ErPa83}]
\label{EPbasic}
\mbox{}
\begin{enumerate}
\item\label{EPbasic,a} For $0 \le \alpha < \frac{1}{2}$, there exists a constant $C(\alpha)$ such that, for each $k \in \mathbb{N}$ and any graph $G$ with at least $C(\alpha) k$ vertices, $G$ or $\overline{G}$ has an induced subgraph $H$ on $\ell \geq k$ vertices with minimum degree at least $\alpha \ell$.
\item\label{EPbasic,b} For $\frac{1}{2} < \alpha \le 1$, there exists a constant $C(\alpha) > 1$ such that, for each $k \in \mathbb{N}$, there is a graph $G$ with at least $C(\alpha)^k$ vertices satisfying the following. If $H$ is any induced subgraph of $G$ or $\overline{G}$ on $\ell \ge k$ vertices, then $H$ has minimum degree less than $\alpha \ell$.
\end{enumerate}
\end{proposition}

Investigating the abrupt change at $\alpha = \frac{1}{2}$, Erd\H{o}s and Pach~\cite{ErPa83} proved the following much stronger result, using graph discrepancy to prove part~\ref{EPbasic,a} and a weighted random graph construction to prove part~\ref{EPbasic,b}.

\begin{theorem}[\cite{ErPa83}] 
\label{EP-variable}
\mbox{}
\begin{enumerate}
\item\label{EP-variable,a} There exists a constant $C>0$ such that, for each $k \in \mathbb{N}$, $k>1$, and any graph $G$ with at least $Ck\ln k$ vertices, $G$ or $\overline{G}$ has an induced subgraph $H$ on $\ell \geq k$ vertices with minimum degree at least $\frac12\ell$.
\item\label{EP-variable,b} For any $\rho \ge 0$, there is a constant $C_\rho > 0$ such that, for large enough $k$, there is a graph $G$ with at least $C_\rho k\ln k/\ln\ln k$ vertices satisfying the following.
If $H$ is any induced graph of $G$ or $\overline{G}$ on $\ell\ge k$ vertices, then $H$ has minimum degree less than $\frac12\ell-\rho$.
\end{enumerate}
\end{theorem}

Our first goal is to further investigate the abrupt change described above. We obtain sharp results by the application of a short discrepancy argument and the analysis of a probabilistic construction similar to Proposition~\ref{EPbasic}\ref{EPbasic,b}.

%\newpage
\begin{theorem}
\label{KPPR-variable}
\mbox{}
\begin{enumerate}
\item\label{KPPR-variable,a} Let $\nu\ge0$ and $c>4/3$ be fixed. For large enough $k$ and any graph $G$ with at least $k^{c10^6\nu^2+4/3}$ vertices, $G$ or $\overline{G}$ has an induced subgraph $H$ on $\ell\ge k$ vertices with minimum degree at least $\frac12(\ell-1) +\nu \sqrt{(\ell-1)\ln \ell}$.
\item\label{KPPR-variable,b} There is a constant $C>0$ such that, if $\nu(\cdot)$ is a non-decreasing non-negative function, then for large enough $k$ there is a graph $G$ with at least $Ck^{\nu(k)^2+1}$ vertices such that the following holds. If $H$ is any induced subgraph of $G$ or $\overline{G}$ on $\ell\ge k$ vertices, then $H$ has minimum degree less than $\frac12(\ell-1) +\nu(\ell) \sqrt{(\ell-1)\ln \ell}$.
\end{enumerate}
\end{theorem}

\noindent

Theorem~\ref{KPPR-variable} exhibits a theshold phenomenon which we elucidate in Section~\ref{sec:thesholds}, where we also make comparisons to Proposition~\ref{EPbasic} and Theorem~\ref{EP-variable}. 
Erd\H{o}s and Pach claimed that their argument for Theorem~\ref{EP-variable}\ref{EP-variable,a} could be extended to prove the statement of Theorem~\ref{KPPR-variable}\ref{KPPR-variable,a} with the term $k^{c10^6\nu^2+4/3}$ replaced by $Ck \ln k$ and $\frac12(\ell-1) +\nu \sqrt{(\ell-1)\ln \ell}$ replaced by $\frac12\ell + \nu\sqrt{\ell}(\ln\ell)^{3/2}$.
Their claimed result contradicts Theorem~\ref{KPPR-variable}\ref{KPPR-variable,b} for $\nu(\ell) = \nu\ln\ell$.

Slightly before the abrupt change occurs, we have found the construction for Theorem~\ref{EP-variable}\ref{EP-variable,b} remains valid, and this yields the following. This improvement is mainly technical in nature, but we have included it for completeness.

\begin{theorem}
\label{thm:weighted}
For any $\nu > 0$, there exists $C_\nu>0$ such that, for large enough $k$, there is a graph $G$ with at least $C_\nu k\ln k/\ln\ln k$ vertices satisfying the following. If $H$ is any induced subgraph of $G$ or $\overline{G}$ on $\ell\ge k$ vertices, then $H$ has minimum degree less than $(\frac12-\ell^{-\nu})(\ell-1)$.
\end{theorem}

\subsection{The fixed quasi-Ramsey problem}

So far, we have discussed the \emph{variable quasi-Ramsey} problem where we seek to guarantee the existence of a clique-like or coclique-like induced subgraph of order at least $k$. It is also natural to ask for such an induced subgraph of order exactly $k$, and we call this the \emph{fixed quasi-Ramsey} problem. In Section~\ref{sec:thinning}, we provide a probabilistic thinning lemma (Lemma~\ref{lem:thinning}) that allows us to translate results about the variable problem into results about the fixed problem. The lemma roughly says that, in any graph of large minimum degree, we can find an induced subgraph of any order that (approximately) preserves the minimum degree condition in an appropriate way. We can use this thinning lemma to establish bounds similar to Proposition~\ref{EPbasic}\ref{EPbasic,a}. We can also use it, together with Theorem~\ref{EP-variable}\ref{EP-variable,a}, to prove the following result.

\begin{theorem}
\label{KPPR-fixed}
There exists a constant $C>0$ such that, for large enough $k$ and any graph $G$ with at least $Ck\ln k$ vertices, $G$ or $\overline{G}$ has an induced subgraph $H$
  on exactly $k$ vertices with minimum degree at least $\frac12(k-1)-2\sqrt{(k-1)\ln k}$.
\end{theorem}

The bound $Ck \ln k$ in Theorem~\ref{KPPR-fixed} is tight up to a $\ln\ln k$ factor by Theorem~\ref{thm:weighted}.
A similar but different result was proved with discrepancy arguments.

\begin{theorem}[\cite{ErPa83}]
\label{EP-fixed}
There exists a constant $C>1$ such that for every $k, \nu \in \mathbb{N}$ and any graph $G$ with at least $C^\nu k^2$ vertices, $G$ or $\overline{G}$ has an induced subgraph $H$ on exactly $k$ vertices with minimum degree at least $\frac12k + \nu$.
\end{theorem}

\subsection{Thresholds and bound comparisons}
\label{sec:thesholds}
We introduce some terminology and notation to facilitate easy comparison of the above results and to exhibit threshold phenomena. A {\em $t$-homogeneous set} is a vertex subset of a graph that induces either a graph of minimum degree at least $t$ or the complement of such a graph.
Let $f : \mathbb{Z}^+ \mapsto \mathbb{N}$ be a non-decreasing non-negative integer function satisfying $f(\ell)<\ell$ for all $\ell$. For any positive integer $k$ the {\em variable quasi-Ramsey number} $R_f(k)$ is defined to be the smallest integer such that any graph of order $R_{f}(k)$ contains an $f(\ell)$-homogeneous set of order $\ell$ for some $\ell \ge k$.
For integers $t$ and $k$ with $0\le t < k$, the {\em fixed quasi-Ramsey number} $R^*_{t}(k)$ is defined to be the smallest integer such that any graph of order $R^*_{t}(k)$ contains a $t$-homogeneous set of order $k$. We refer to both $R_{f}(k)$ and $R^*_t(k)$ as {\em quasi-Ramsey numbers}. Versions of these parameters were introduced in~\cite{ErPa83}. 

Note that Proposition~\ref{EPbasic} shows that, for any fixed $\eps > 0$, as $f$ changes from a function satisfying $f(\ell) \leq (\frac{1}{2} - \eps) \ell$ for all $\ell$ to a function satisfying $f(\ell) \geq (\frac{1}{2} + \eps) \ell$ for all $\ell$, $R_{f}(k)$ changes from  polynomial (indeed, linear) growth in $k$ to superpolynomial (indeed, exponential) growth in $k$. Theorem~\ref{EP-variable}\ref{EP-variable,a} narrows this gap by showing that we can replace $(\frac{1}{2} - \eps)\ell$ above with $\frac{1}{2}\ell$ to achieve polynomial growth in $k$. Theorem~\ref{KPPR-variable} shows that as $f$ changes from a function satisfying $f(\ell) \leq \frac12\ell+ o(\sqrt{\ell\ln \ell})$ for all $\ell$ to a function satisfying $f(\ell) \geq \frac12\ell+ \omega(\sqrt{\ell\ln \ell})$ for all $\ell$, $R_{f}(k)$ changes from  polynomial growth in $k$ to superpolynomial growth in $k$.

The fixed quasi-Ramsey numbers are less well understood. Theorem~\ref{KPPR-fixed} shows that $R^*_t(k)=O(k \ln k)$ for $t \leq \frac{1}{2}k - \omega(\sqrt{k \ln k})$, while Theorem~\ref{EP-fixed} shows that $R^*_t(k)=O(k^2)$ for $t = \frac{1}{2}k + O(1)$. Since $R^*_{f(k)}(k) \geq R_f(k)$, Theorem~\ref{KPPR-variable}\ref{KPPR-variable,b} implies that $R^*_t(k)$ is superpolynomial in $k$ if $t \geq \frac12 k+ \omega(\sqrt{k\ln k})$.

\subsection{Further related work}

We mention work on the fixed quasi-Ramsey problem by Chappell and Gimbel~\cite{ChGi11}.
Using an Erd\H{o}s--Szekeres-type recursion, they 
proved for $t \ge 1$ that
\begin{align*}
R^*_t(k) \le (k-t-1)\binom{2(t-1)}{t-1}+\binom{2t}{t}
 \le (k-t+3)4^{t-1}.
\end{align*}
They gave an exact formula for $R^*_t(k)$ when $t$ is small: if $1\le t \le \frac14(k+2)$, then $R^*_t(k) = k+2t-2$.  They also showed 
the lower bound of $k+2t-2$ holds for all $t\le \frac12(k+1)$; 
a construction certifying this is depicted in Figure~\ref{fig:lowerex}.

\begin{figure}[ht]
\centering
\begin{tikzpicture}[scale=1.6]
    \tikzstyle{ann} = [draw=none,fill=none,right]
     \filldraw[fill=black!25](-0.65,0.2)--(-0.65,-0.2)--(0.65,-0.2)--(0.65,0.2)--cycle;
     \filldraw[fill=black!70](-0.8,-0.15)--(-0.5,0.15)--(0.15,-0.85)--(-0.15,-1.15)--cycle;
     \draw (-0.65,0) node[draw,ellipse,minimum height=1.6cm,minimum width=1.4cm,fill=black!70] {$P$};
     \draw (0.65,0) node[draw,ellipse,minimum height=1.6cm,minimum width=1.4cm,fill=black!0] {$Q$};
     \draw (0,-1) node[draw,ellipse,minimum height=1.4cm,minimum width=1.6cm,fill=black!0] {$R$};
\end{tikzpicture}
\caption{\small{An illustration of the construction by Chappell and Gimbel that gives $R_t(k) \ge k+2t-2$ for all $t\le \frac12(k+1)$.  In this example, $P$ is a clique of order $2(t-1)$, $Q$ is a coclique of order $2(t-1)$, $R$ is a coclique of order $k-2t+1$, all possible edges between $P$ and $R$ are present, all possible edges between $Q$ and $R$ are absent, and the bipartite subgraph induced by the edges between $P$ and $Q$ is $(t-1)$-regular.  (Note that the subgraph on $R$ could instead be chosen arbitrarily.)}}
\label{fig:lowerex}
\end{figure}
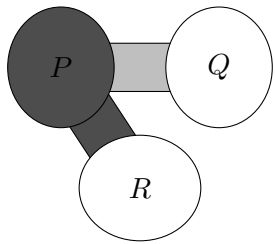

\paragraph{Notation.}
Chappell and Gimbel chose the complementary interpretation for $R^*_t(k)$ --- so the sets of order $k$ have {\em maximum} degree bounded by $t$ --- and referred to the parameters as {\em defective Ramsey numbers}.  Our $R^*_t(k)$ is essentially the same as what is $R^*_{t/k}(k)$ in the notation of Erd\H{o}s and Pach, while our $R_{f}$ 
slightly refines their $R_{\alpha}$ allowing for more precise statements.

\paragraph{Structure of the paper.}
We prove Theorem~\ref{KPPR-variable}\ref{KPPR-variable,a} in Section~\ref{sec:discrepancy}.
We prove Theorem~\ref{KPPR-variable}\ref{KPPR-variable,b} and discuss related results in Section~\ref{sec:lower}.
We state and prove the thinning approach and discuss its applications, such as Theorem~\ref{KPPR-fixed}, in Section~\ref{sec:thinning}.
In Section~\ref{sec:weighted}, we 
prove Theorem~\ref{thm:weighted}.
We give some concluding remarks and prompt some questions for further investigation in Section~\ref{sec:conclusion}.

\paragraph{Acknowledgement.}
We thank John Gimbel for stimulating discussions as well as for sending us the manuscript~\cite{ChGi11}.

%%%%%%%%%%%%%%%%%%%%%%%%%%%%%%%%%%%%%%%%%%%%%%%%%%%%%%%%%%%%%%%%%%%%%%%%%

\section{An upper bound using discrepancy}
\label{sec:discrepancy}

We use a result on graph discrepancy to prove Theorem~\ref{KPPR-variable}\ref{KPPR-variable,a}. 
Given a graph $G=(V,E)$, the {\em discrepancy} of a set $X\subseteq V$ is defined as
\begin{align*}
D(X):= e(X)-\frac{1}{2}\binom{|X|}{2},
\end{align*}
where $e(X)$ denotes the number of edges in the subgraph $G[X]$ induced by $X$.
We use a result of Erd\H{o}s and Spencer~\cite[Ch.~7]{ErSp74}, which is the same result used by Erd\H{o}s and Pach for
their proof of Theorem~\ref{EP-variable}\ref{EP-variable,a}.
\begin{lemma}[Theorem~7.1 of~\cite{ErSp74}]\label{lem:ES}
Provided $n$ is large enough, if $t \in\{1,\dots,n\}$, then any graph $G=(V,E)$ of order $n$ satisfies
\begin{align*}
\max _{S\subseteq V, |S|\le t} |D(S)|\ge \frac{t^{3/2}}{10^3}\sqrt{\ln(5n/t)}.
\end{align*}
\end{lemma}

\begin{proof}[Proof of Theorem~\ref{KPPR-variable}\ref{KPPR-variable,a}]
Fix $\nu\ge0$ and let $G=(V,E)$ be any graph on at least $N =k^{c10^6\nu^2+4/3}$ vertices. For the theorem, it suffices to prove that $G$ or $\overline{G}$ has an induced subgraph $H$ on $\ell\ge k$ vertices with minimum degree at least $\frac12(\ell-1) +\nu \sqrt{(\ell-1)\ln \ell}$. For any $X\subseteq V$, we define the following skew form of discrepancy:
\begin{align*}
D_\nu(X):= |D(X)| - \nu\sqrt{|X|^3\ln|X|}.
\end{align*}
 
Let $X\subseteq V$ be a set attaining maximum skew discrepancy. 
By symmetry we may assume that $D(X)>0$.
Then for any $x\in X$ we have 
\begin{equation}
\deg_{G[X]}(x)\ge \tfrac12(|X|-1) +\nu \sqrt{|X|\ln|X|}.\label{eq:deg}
\end{equation}
To see \eqref{eq:deg}, suppose $x\in X$ has strictly smaller degree than claimed and set $X':=X\setminus \{x\}$.
Then
\begin{align*}
D_\nu(X')&\ge e(X')-\frac{1}{2}\binom{|X|-1}{2} -\nu\sqrt{(|X|-1)^3\ln(|X|-1)}	\nonumber
\\
&>e(X)-\frac{1}{2}\binom{|X|}{2} -\nu \sqrt{|X|\ln |X|}-\nu\sqrt{(|X|-1)^3\ln(|X|-1)}.	
\end{align*}
Note that $\sqrt{|X|^3\ln |X|}>\sqrt{|X|\ln |X|}+\sqrt{(|X|-1)^3\ln(|X|-1)}$,
which by the above  
implies $D_\nu(X')>D_\nu(X)$, contradicting the maximality of $D_\nu(X)$.

If $k$ is large enough, then by Lemma \ref{lem:ES} there exists a set of at most $k^{4/3}$ vertices with discrepancy at least $\nu k^2\sqrt{c\ln k}$. Recall that $c>4/3$. So, in evaluating this set's skew discrepancy, the ordinary discrepancy term will dominate the `skew term', which is equal to $-\nu k^2\sqrt{4/3\ln k}$.
We may thus assume that $D_\nu(X)\geq k^2$ if $k$ is large enough, but now note that this implies that $|X| \ge k$, as required.
\end{proof}

This argument is considerably shorter than Erd\H{o}s and Pach's proof of Theorem~\ref{EP-variable}\ref{EP-variable,a}. If we follow the original approach more closely, then after appropriate adjustments we can obtain a slight improvement upon Theorem~\ref{KPPR-variable}\ref{KPPR-variable,a} whereby $k^{c10^6\nu^2+4/3}$ with $c>4/3$ is replaced by $200(k\ln k)^{c10^6\nu^2+1}$ with $c>1$. Note that with $\nu=0$ this results in a bound akin to Theorem~\ref{EP-variable}\ref{EP-variable,a}. For clarity of exposition, we elected for the shorter argument, which still yields the threshold phenomenon we desire.

%%%%%%%%%%%%%%%%%%%%%%%%%%%%%%%%%%%%%%%%%%%%%%%%%%%%%%%%%%%%%%%%%%%%%%%%%

\section{Random graph lower bounds}
\label{sec:lower}

Next we give probabilistic lower bounds for the quasi-Ramsey numbers. 
We elaborate on an observation by Erd\H{o}s and Pach.
We apply upper bounds on the order of largest $t$-homogeneous sets in random graphs to extend the classic lower bounds on $R(k)$~\cite{Erd47,Spe75}.
We rely on analysis from~\cite{KaMc10}, which amongst other things thoroughly describes the expected behaviour of {\em $t$-dependent sets} --- i.e.~vertex subsets that induce subgraphs of maximum degree at most $t$ --- in the random graph $\G{n}{1/2}$ with vertex set $[n] = \{1,\dots,n\}$ and edge probability $\frac12$.
We need a result best stated with large deviations notation.  For more on large deviations, consult~\cite{DeZe98}.  Let
\begin{align*}
\Lambda^*(x) = 
\begin{cases}
x \ln (2x) + (1 - x) \ln (2(1 - x)) & \mbox{for $x\in [0, 1]$}\\
\infty & \mbox{otherwise}
\end{cases}
\end{align*}
(where $\Lambda^*(0) = \Lambda^*(1) = \ln 2$).
This is the Fenchel-Legendre transform of the logarithmic moment generating function associated with the Bernoulli distribution with probability $\frac12$ (cf.~Exercise 2.2.23(b) of~\cite{DeZe98}).  Some easy calculus checks that $\Lambda^*(x)$ has a global minimum of $0$ at $x = \frac12$, is strictly decreasing on $[0, \frac12)$ and strictly increasing on $(\frac12,1]$.  The following bounds the probability that a given subset of order $k$ in $\G{n}{1/2}$ is $t$-dependent.

\begin{lemma}[Lemma~2.2(i) of~\cite{KaMc10}]
\label{lem:A_n}
Given $\bar{t},k$ with $\bar{t} \le \frac12(k - 1)$,
\begin{align*}
\Pr(\Delta(\G{k}{1/2}) \le \bar{t}) \le \exp\left(-\binom{k}{2} \Lambda^*\left( \frac{\bar{t}}{k - 1} \right) \right).
\end{align*}
\end{lemma}

\begin{proposition} \label{prop:lower}
For any $\eps\ge0$ let $f(\ell)$ be any function satisfying $f(\ell) \ge (\frac{1}{2} + \eps)(\ell-1)$ for all $\ell$. Then, as $k\to\infty$,
\begin{align*}
R_{f}(k) \ge (1+o(1)) \frac{k}{e}\exp\left(\frac{k-1}2\Lambda^*\left(\frac12-\eps\right)\right).
\end{align*}
\end{proposition}
\begin{proof}
For any $\delta>0$ and some large enough integer $k$, let
\begin{align*}
n = \left\lfloor \frac1{1+\delta}\frac{k}{e}\exp\left(\frac{k-1}2\Lambda^*\left(\frac12-\eps\right)\right) \right\rfloor.
\end{align*}
Consider the random graph $G \sim \G{n}{1/2}$. Given a subset $S \subseteq [n]$ of $\ell\ge k$ vertices, let $A_S$ be the event that $\delta(G[S])\ge f(\ell)$ or $\delta(\overline{G}[S])\ge f(\ell)$, where $\delta(\cdot)$ denotes the minimum degree of the graph.
Since $\eps\ge0$, we have by Lemma~\ref{lem:A_n} that
\begin{align*}
\Pr(A_S)
& \le 2\exp\left(-\binom{\ell}{2}\Lambda^*\left(\frac{\ell-f(\ell)-1}{\ell-1}\right)\right)
 \le 2\exp\left(-\binom{\ell}{2}\Lambda^*\left(\frac12-\eps\right)\right).
\end{align*}
Note that we allow the possibility that $\eps > 1/2$, in which case the above inequality gives $\Pr(A_S) \leq 0$.
So the probability that $A_S$ holds for some set $S \subseteq [n]$ of $\ell\ge k$ vertices is at most 
\begin{align*}
\sum_{S \subseteq [n], |S| \ge k} \Pr(A_S)
& \le \sum_{\ell = k}^n \binom{n}{\ell} 2\exp\left(-\binom{\ell}{2}\Lambda^*\left(\frac12-\eps\right)\right)\\ 
& \le 2\sum_{\ell = k}^n \left(\frac{en}\ell\cdot \exp\left(-\frac{\ell-1}{2}\Lambda^*\left(\frac{1}{2}-\eps\right)\right) \right)^{\ell} \\
& \le 2\sum_{\ell=k}^n (1+\delta)^{-\ell}< 1,
\end{align*}
where in this sequence of inequalities we have used the definition of $n$, the fact that $\ell\ge k$, and a choice of $k$ large enough.
Thus, for $k$ large enough, there exists a graph on $n$ vertices for which each induced subgraph of order $\ell\ge k$ and its complement have minimum degree less than $f(\ell)$. Since we proved this statement holds for any $\delta>0$, the result follows.
\end{proof}

As we see now, Theorem~\ref{KPPR-variable}\ref{KPPR-variable,b} follows the same argument.

\begin{proof}[Proof of Theorem~\ref{KPPR-variable}\ref{KPPR-variable,b}]
Into the proof of Proposition~\ref{prop:lower}, we substitute
\begin{align*}
\eps = \eps(\ell) = \nu(\ell) \sqrt{\frac{\ln \ell}{\ell-1}}.
\end{align*}
By the Taylor expansion of $\Lambda^*$ (for $0 \leq \eps \leq 1/2$), we have that
\begin{align*}
\Lambda^*\left(\frac12-\eps\right) 
& = \left(\frac12-\eps\right)\ln(1-2\eps) + \left(\frac12+\eps\right)\ln(1+2\eps) \\
& = \sum_{j=1}^\infty\frac{(2\eps)^{2j}}{2j(2j-1)} \ge 2\eps^2.
\end{align*}
Note that $\Lambda^*\left(\frac12-\eps\right) \geq 2\eps^2$ in fact
holds for all $\eps \geq 0$. 
Now for any $\delta>0$ let
\begin{align*}
n = \left\lfloor \frac{1}{1+\delta} \frac{k^{\nu(k)^2+1}}e \right\rfloor,
\end{align*}
where $k$ is some large enough integer.
Again consider the random graph $G \sim \G{n}{1/2}$. Let $f(\ell) = (\frac{1}{2} + \eps(\ell))(\ell-1)$ and $A_S$ be as in Proposition~\ref{prop:lower}. As we did before, but also using the Taylor expansion above, we obtain that the probability $A_S$ holds for some set $S \subseteq [n]$ of $\ell\ge k$ vertices is at most 
\begin{align*}
\sum_{S \subseteq [n], |S| \ge k} \Pr(A_S)
& \le 2\sum_{\ell = k}^n \left(\frac{en}\ell\cdot \exp\left(-(\ell-1)\eps^2\right) \right)^{\ell}
 = 2\sum_{\ell = k}^n \left(\frac{en}{\ell^{\nu(\ell)^2+1}} \right)^{\ell} \\
& \le 2\sum_{\ell=k}^n (1+\delta)^{-\ell}< 1,
\end{align*}
by the choice of $n$, $\ell\ge k$, $\nu(\ell)\ge\nu(k)$, and $k$ large enough.
Thus, for $k$ large enough, there is a graph on $n$ vertices where each induced subgraph of order $\ell\ge k$ and its complement have minimum degree less than $f(\ell)$. This holds for any $\delta>0$, so the result follows.
\end{proof}

For the fixed quasi-Ramsey numbers $R^*_t(k)$, we can get a constant factor improvement upon the bound implied by Proposition~\ref{prop:lower} by additionally using the Lov\'asz Local Lemma 
as Spencer~\cite{Spe75} did for $R(k)$. In particular, for $t = t(k) \ge (\frac12+\eps)(k-1)$, the factor is $\exp\left(\Lambda^*\left(\frac12-\eps\right)\right)$. This is standard and the calculations are similar to those used above, so we omit the proof.  

 \begin{proposition}
 \label{prop:lower,spencer}
 For $\eps \ge 0$, let $t = t(k) \ge (\frac12+\eps)(k-1)$.  Then, as $k\to\infty$,
 \begin{align*}
 R^*_t(k) \ge (1+o(1))\frac{k}{e}\exp\left(\frac{k+1}2 \Lambda^*\left(\frac12-\eps\right)\right).
 \end{align*}
 \end{proposition}

%%%%%%%%%%%%%%%%%%%%%%%%%%%%%%%%%%%%%%%%%%%%%%%%%%%%%%%%%%%%%%%%%%%%%%%%%

\section{A thinning argument for upper bounds}
\label{sec:thinning}
We start this section by explicitly stating our thinning approach.
\begin{lemma}
\label{lem:thinning}
For any $0<c<1$ and $\eps>0$, let $k$ be such that
\begin{align}
\label{eqn:K}
\exp\left(\frac12\eps^2(k-1)\right) > k.
\end{align}
If $H$ is a graph of order $\ell \ge k$ such that $\delta(H) \ge c\ell$,  then there exists $S \subseteq V(H)$ of order $k$ such that $\delta(H[S]) \ge (c-\eps)(k-1)$.
\end{lemma}

For Lemma~\ref{lem:thinning}, we require a Chernoff-type bound for the hypergeometric distribution.
Given positive integers $N$, $b$, $a$ with $a,b \le N$, choose $S \subseteq [N]$ with $|S|=b$ uniformly at random (u.a.r.).  The random variable given by $X = |S \cap [a]|$ is a hypergeometric random variable with parameters $N$, $b$, $a$.

\begin{lemma}[Theorem~2.10 and (2.6) of~\cite{JLR00}] \label{lem:Chernoff}
If $X$ is a hypergeometric random variable with parameters $N$, $b$, $a$, and $d \ge 0$, then
\[
\Pr \left( X \le \frac{ab}{N} - d \right) \le \exp\left( -\frac{d^2N}{2ab} \right).
\]                      
\end{lemma}

\begin{proof}[Proof of Lemma~\ref{lem:thinning}]
Assume $c$, $\eps$, $k$, and $H$ are as in the statement of the lemma.
Given a vertex $v \in V(H)$ and a subset $T \subseteq V(H)\setminus\{v\}$ of order $k-1$, we call $(v,T)$ a {\em pair}.
We say that a pair $(v,T)$ is {\em good} if $\deg_T(v) \ge (c-\eps)(k-1)$; otherwise it is {\em bad}.
Given a subset $U \subseteq V(H)$ of order $k$, we say it is {\em good} if $(w,U\setminus \{w\})$ is good for all $w\in U$; otherwise it is {\em bad}. 

Note that if we can find a good $U$ in $H$, then we are done.  Also observe that if $U$ is bad for all $U \subseteq V(H)$ of order $k$, then there must be at least $\binom{\ell}{k}$ distinct bad pairs.
However, there are $\ell\binom{\ell-1}{k-1}$ pairs in total.
So there exists a good $U$ provided that, when choosing a pair $(v,T)$ u.a.r.,
\begin{align*}
\Pr((v,T)\text{ is bad}) < \left.\binom{\ell}k\right/\ell\binom{\ell-1}{k-1} = \frac1k.
\end{align*}
We pick $(v,T)$ u.a.r.~by choosing $v$ u.a.r.~before choosing $T$ u.a.r.  Note that, given $v$ and a uniform choice of subset $T\subseteq V(H)\setminus\{v\}$ of order $k-1$, the random variable $\deg_T(v)$ has a hypergeometric distribution with parameters $\ell$, $k-1$, $\deg(v)$.  Since $c\ell\le\deg(v)\le\ell$, we have by Lemma~\ref{lem:Chernoff} that
\begin{align*}
\Pr((v,T)\text{ is bad} \ | \ v)
& = \Pr(\deg_T(v) < (c-\eps)(k-1)) \\
& \le \Pr\left(\deg_T(v) < \frac{\deg(v)(k-1)}\ell - \eps(k-1)\right) \\
& \le \exp\left(-\frac{\eps^2(k-1)\ell}{2\deg(v)}\right)
 \le \exp\left(-\frac12\eps^2(k-1)\right).
\end{align*}
By~\eqref{eqn:K}, the last quantity is less than $1/k$ so
it follows that
\begin{align*}
\Pr((v,T)\text{ is bad})
& = \frac1\ell\sum_v\Pr((v,T)\text{ is bad} \ | \ v) < \frac1k,
\end{align*}
as desired.
\end{proof}

\noindent
Our first application of the thinning lemma is the following upper bound for $R^*_t(k)$. This complements the bounds of Chappell and Gimbel mentioned in the introduction. Since it is not close to the lower bound, we did not attempt to optimise it, though it can easily be improved to roughly $(\eps^{-1/2}\sqrt{1/2+\eps})\cdot k$.

\begin{theorem} \label{thm:upper}
Let $\eps > 0$.
If $t = t(k) \le (\frac12-\eps)(k-1)$, then
\begin{align*}
R^*_t(k) \le \eps^{-1/2}\sqrt{1+\eps}\cdot(k+o(k)).
\end{align*}
\end{theorem}

\begin{proof}
Choose $k$ large enough so that it satisfies~\eqref{eqn:K} with $\eps$ halved, and
let $G$ be a graph of order
$n \ge \eps^{-1/2}\sqrt{1+\eps}(k+\gamma k)$
for some small fixed $\gamma > 0$.
By considering $G$ or its complement, we may assume without loss of generality that $|E(G)| \ge \frac12\binom{n}2$.
We require the following explicit form of Theorem~\ref{EPbasic}\ref{EPbasic,a}. This is essentially given as Exercise~12.8 in~\cite{Bol01}, so we omit the proof. (The idea is to repeatedly remove any vertex of too small degree.)

\begin{lemma}
Let $0 \le \alpha < 1/2$ and suppose that
\begin{align*}
n \ge \frac{\sqrt{1-\alpha}}{\left(\tfrac12-\alpha\right)^{1/2}}\cdot k \cdot\left(1+\frac1{k(1-\alpha)\left(\tfrac12-\alpha\right)^{1/2}}\right)^{1/2}
\end{align*} 
(for $k$ chosen large enough). 
If $G$ is a graph with $|V(G)| = n$ and $|E(G)| \ge \frac12\binom{n}{2}$,
then it has a subgraph $H$ of order at least $k$ such that $\delta(H) \ge \alpha|V(H)|$.
\end{lemma}

For large enough $k$, our choice of $n$ satisfies the hypothesis of the lemma with $\alpha = \frac12(1-\eps)$. So we are guaranteed a subgraph $H$ with $|V(H)| \ge k$ and $\delta(H) \ge \frac12(1-\eps)|V(H)|$.
By Lemma~\ref{lem:thinning} with $c = \frac12(1-\eps)$ and $\eps$ halved, there exists $S \subseteq V(H) \subseteq V(G)$ of order $k$ with $\delta(G[S]) \ge (\frac12-\eps)(k-1)$.
\end{proof}

We also apply our thinning lemma to prove 
Theorem~\ref{KPPR-fixed}.

\begin{proof}[Proof of Theorem~\ref{KPPR-fixed}]
Let $G$ be a graph of order $C k \ln k$, where $C$ is the same constant as in Theorem~\ref{EP-variable}\ref{EP-variable,a}. 
Then $G$ or $\overline{G}$ contains a subgraph $H$ of order $\ell$, where $k\le \ell\le Ck\ln k$, with $\delta(H) \ge \frac12\ell$. 
For the application of Lemma~\ref{lem:thinning}, set $c = \frac12$ and $\eps = \sqrt{2\ln(k+1)/(k-1)}$.
Then
$\exp\left(\frac12\eps^2(k-1)\right) =k+1
> k$,
and so Lemma~\ref{lem:thinning}
yields a set $S\subseteq V(H)\subseteq V(G)$ of order $k$ such that
\begin{align*}
\delta(G[S]) = \delta(H[S]) 
 \ge \left(\frac12 - \eps\right)(k-1)
 \ge \frac12(k-1)-2\sqrt{(k-1)\ln k},
\end{align*}
which proves the theorem.
\end{proof}

%%%%%%%%%%%%%%%%%%%%%%%%%%%%%%%%%%%%%%%%%%%%%%%%%%%%%%%%%%%%%%%%%%%%%%%%%

\section{A weighted random graph construction}
\label{sec:weighted}

In this section, by a careful analysis of the weighted construction that  Erd\H{o}s and Pach used for Theorem~\ref{EP-variable}\ref{EP-variable,b}, we extend the validity of that result, thereby establishing Theorem~\ref{thm:weighted}.

\begin{proof}[Proof of Theorem~\ref{thm:weighted}]
By the monotonicity of $R_{t(\ell)}$, there is no loss of generality in assuming $\nu <\frac27$. Let $\nu' = \frac12\nu+\frac17$.
Let $k$ be some sufficiently large integer.
Let $g(\cdot)$ be the function defined by
\begin{align*}
g(x) = \left\lfloor\frac{\nu'}{8}\frac{\ln x}{\ln\ln x}\right\rfloor
\end{align*}
and write $z = g(k)$.
Construct a graph $G = (V,E)$ randomly as follows.  The vertex set is defined
$V = V_1 \cup \dots \cup V_z$,
for disjoint sets $V_1,\dots,V_z$ with
\begin{align*}
|V_1| = \dots = |V_z| = \left\lfloor\left(1-\frac{1}{2z}\right)k\right\rfloor.
\end{align*}
Note that $|V| < k\ln k$ and
\begin{align*}
|V| \ge z(k-1) \ge \frac{\tfrac12\nu+\tfrac17}{10}\cdot\frac{k\ln k}{\ln\ln k}.
\end{align*}
Thus we can safely choose $C_\nu=\nu/20$ for the statement of the theorem.
The random edge set $E$ of $G$ is determined according to a skewed distribution. Given vertices $v_i\in V_i$ and $v_j\in V_j$, the probability of their being joined by an edge is defined by
\begin{align*}
\Pr(v_iv_j\in E) = p_{ij} = 
\begin{cases}
\frac12 - (2z)^{-4(i+j)-1} & \text{ if $i\ne j$;}\\
\frac12 + (2z)^{-8i} & \text{ if $i=j$.}
\end{cases}
\end{align*}

The remainder of the proof is devoted to proving that $G$ has the properties we desire with positive probability. Let $X$ be an arbitrary subset of $\ell \ge k$ vertices and for convenience write $\ell_i = |X\cap V_i|$ for every $i\in\{1,\dots,z\}$. We will show that $X$ is $t$-homogeneous with very small probability, where
$t = (\tfrac12-\hat{\eps})(\ell-1)$
for some $\hat{\eps}=\hat{\eps}(\ell)>0$ to be specified later.

First we concentrate on the minimum degree of the graph $G[X]$ induced by $X$.
To this end, let $j'$ be the largest integer that satisfies $\ell_{j'} \ge \ell/(4z^2)$, so that $\ell_i < \ell/(4z^2)$ for all $i > j'$. By this choice of $j'$, note that
\begin{align}
\sum_{i<j'}\ell_i 
& \ge \ell - |V_{j'}| -\frac{z\ell}{4z^2}
\ge \ell - \left(1-\frac{1}{2z}\right)k -\frac{\ell}{4z}\nonumber\\
& = \left(1-\frac{1}{4z}\right)(\ell - k) +\frac{k}{4z}
\ge \frac{\ell}{4z},\nonumber
\end{align}
for large enough $k$. We consider the minimum degree only among vertices in $X\cap V_{j'}$. Let $v\in X\cap V_{j'}$. Since the degree of $v$ in $G[X]$ is the sum $\sum_i e(v, X\cap V_i)$ (where $e(v,S)$ denotes the number of edges between $v$ and $S$), its expectation satisfies
\begin{align}
\Exp&(\deg_{G[X]}(v))
 = (\ell_{j'}-1)p_{j'j'}+\sum_{i\ne j'}\ell_ip_{ij'} \nonumber\\
& = (\ell_{j'}-1)\left(\frac12+\frac{1}{(2z)^{8j'}}\right) + \sum_{i\ne j'}\ell_i\left(\frac12-\frac{1}{(2z)^{4(i+j')+1}}\right) \nonumber\\
& \le \frac12(\ell-1)+\frac{\ell_{j'}-1}{(2z)^{8j'}}-\sum_{i< j'}\frac{\ell_i}{(2z)^{4(i+j')+1}} \nonumber\\
& \le \frac12(\ell-1)+\frac{\ell_{j'}-1}{(2z)^{8j'}}-\frac{\ell(2z)^{2}}{2(2z)^{8j'}}
 \le \left(\frac12-\frac{1}{(2z)^{8z}}\right)(\ell-1), \nonumber
\end{align}
for large enough $k$.
We also easily have that $\Exp(\deg_{G[X]}(v))\ge\frac13(\ell-1)$.
Since $\deg_{G[X]}(v)$ is a sum of independent Bernoulli random variables, it follows by Hoeffding's inequality (cf.~\cite[Eq.~(2.14)]{JLR00}) that, for any $\eps>0$, provided $k$ is large enough,
\begin{align}
\Pr&(\deg_{G[X]}(v) > (1+\eps)\Exp(\deg_{G[X]}(v))) \nonumber\\
&< \exp(-\tfrac13\eps^2\Exp(\deg_{G[X]}(v)))
\le \exp(-\tfrac19\eps^2(\ell-1)).
\label{eqn:degreelower}
\end{align}
Although this bound is already quite small, for our purposes we require an even stronger bound on $\Pr(\delta(G[X]) > (1+\eps)\Exp(\deg_{G[X]}(v)))$.
For this, we restrict our attention further by bounding the minimum degree among vertices of some arbitrary subset $Y\subseteq X\cap V_{j'}$ of order $\frac12 \eps \ell_{j'}$.
Now if $v\in Y$ has degree in $G[X]$ greater than $(1+\eps)\Exp(\deg_{G[X]}(v))$, then the number of neighbours of $v$ outside $Y$ must be greater than $(1+\frac12\eps)\Exp(\deg_{G[X]}(v)))$. Note that the random variables that count the number of neighbours of $v$ in $G[X]$ outside $Y$ for all $v\in Y$ are mutually independent. Also, since $Y$ is small, the following analogue of~\eqref{eqn:degreelower} holds for each $v\in Y$, as long as $k$ is large enough:
\begin{align*}
\Pr(e(v,X\setminus Y) > (1+\tfrac12\eps)\Exp(\deg_{G[X]}(v)))
< \exp(-\tfrac1{36}\eps^2(\ell-1)).
\end{align*}
Combining these observations, it follows, for any $\eps>0$, that if $k$ is sufficiently large then 
\begin{align}
\Pr&\left(\delta(G[X]) > (1+\eps)\left(\frac12-\frac{1}{(2z)^{8z}}\right)(\ell-1)\right) \nonumber\\
& \le \Pr(\delta(G[X]) > (1+\eps)\Exp(\deg_{G[X]}(v))) \nonumber\\
& \le \Pr(\forall v\in Y: \deg_{G[X]}(v) > (1+\eps)\Exp(\deg_{G[X]}(v))) \nonumber\\
& \le \prod_{v\in Y} \Pr(e(v,X\setminus Y) > (1+\tfrac12\eps)\Exp(\deg_{G[X]}(v))) \nonumber\\
& \le \exp(-\tfrac1{72}\eps^3(\ell-1)\ell_{j'}) \le \exp\left(-\frac{\eps^3(\ell-1)\ell}{288z^2}\right).\label{eqn:j'}
\end{align}

Next we concentrate on the minimum degree of the complement of $G[X]$. To this end, let $j^*\in\{1,\dots,z\}$ be such that
$\ell_{j^*}(2z)^{-4j^*}$
is maximised.
By an averaging argument, this choice of $j^*$ implies
\begin{align*}
\frac{\ell_{j^*}}{(2z)^{4j^*}} \ge
 \frac{\ell}{z(2z)^{4z}} \ge 
 \frac{\ell -1 }{z(2z)^{4z}}.
\end{align*}
 We shall consider the maximum degree only among vertices in $X\cap V_{j^*}$. Let $v\in X\cap V_{j^*}$. Then we have that the expected degree of $v$ in $G[X]$ satisfies for all large enough $k$ that
\begin{align*}
\Exp&(\deg_{G[X]}(v))
 = (\ell_{j^*}-1)p_{j^*j^*}+\sum_{i\ne j^*}\ell_ip_{ij^*} \\
& = (\ell_{j^*}-1)\left(\frac12+\frac{1}{(2z)^{8j^*}}\right) + \sum_{i\ne j^*}\ell_i\left(\frac12-\frac{1}{(2z)^{4(i+j^*)+1}}\right)\\
& \ge \frac12(\ell-1) + \frac{\ell_{j^*}-1}{(2z)^{8j^*}}
- \frac{\ell_{j^*}(z-1)}{(2z)^{8j^*+1}}
 \ge \frac12(\ell-1) + \frac{\ell_{j^*}}{2(2z)^{8j^*}} \\
& \ge \left(\frac12 + \frac{1}{(2z)^{8z+1}}\right)(\ell-1).
\end{align*}
We also easily see that $\Exp(\deg_{G[X]}(v)) \le \frac23(\ell-1)$. By similar arguments as above, but for the complement $\overline{G}$ of $G$, we obtain, for any $\eps > 0$, that if $k$ is large enough then
\begin{align}
\Pr\left(\delta(\overline{G}[X]) > (1+\eps)\left(\tfrac12 - (2z)^{-8z-1}\right)(\ell-1)\right)
 \le \exp\left(-\frac{\eps^3(\ell-1)\ell}{72z(2z)^{4z}}\right).\label{eqn:j*}
\end{align}

To tie everything together, we apply~\eqref{eqn:j'} and~\eqref{eqn:j*} with a common choice of $\eps$. In particular, let $\hat{\eps}(\cdot)$ be the function defined by
\begin{align*}
\hat{\eps}(x) = (2g(x))^{-8g(x)-2}
\end{align*}
and let $\eps = \hat{\eps}(k\ln k)$.
Note that since $k\le\ell \le |V| < k\ln k$ we have that $\eps < \hat{\eps}(\ell)\le\hat{\eps}(k)$.
By our definition of $g(\cdot)$, we obtain that as $k\to \infty$ both
\begin{align*}
\hat{\eps}(k) \sim k^{-(1+o(1))\nu'}
\text{ and }
\eps \sim k^{-(1+o(1))\nu'},
\end{align*}
so that $\hat{\eps}(\ell) > \ell^{-\nu}$ for large enough $k$, by the choice of $\nu'$.
Also, for large $k$,
\begin{align*}
(1+\eps)\left(\tfrac12-(2z)^{-8z}\right) 
\le
(1+\eps)\left(\tfrac12-(2z)^{-8z-1}\right)\\
\le
\tfrac12-\hat{\eps}(k)
\le
\tfrac12-\hat{\eps}(\ell)
\le
\tfrac12-\eps.
\end{align*}
Then, by~\eqref{eqn:j'} and~\eqref{eqn:j*}, the probability that the set $X$ is $((\frac12-\hat{\eps}(\ell))(\ell-1))$-homogeneous is, for all $k$  sufficiently large, at most
\begin{align*}
2&\exp\left(-\frac{\eps^3(\ell-1)\ell}{288z(2z)^{4z}}\right)
\le 2\exp\left(-\frac{\ell(\ell-1)}{144(2g(k\ln k))^{28g(k\ln k)+7}}\right) \\
& = 2\exp(-\ell^{2-(1+o(1))7\nu'/2})
< 2\exp(-k^{2-(1+o(1))7\nu'/2}).
\end{align*}

The above estimate holds for any $X$ with $\ell\ge k$ vertices. Thus the probability that $G$ has a $((\frac12-\hat{\eps}(\ell))(\ell-1))$-homogeneous set with $\ell\ge k$ vertices is less than
\begin{align*}
2^{zk}\cdot 2\exp(-k^{2-(1+o(1))7\nu'/2}),
\end{align*}
which is less than $1$ for $k$ large enough, since $z=k^{o(1)}$ and $\nu'<\frac27$. For large enough $k$ we have $\hat{\eps}(\ell) > \ell^{-\nu}$, and so conclude there is a graph of order at least $C_\nu k\ln k/\ln\ln k$
in which no vertex subset of order $\ell\ge k$ is $((\frac12-\ell^{-\nu})(\ell-1))$-homogeneous, as required.
\end{proof}

%%%%%%%%%%%%%%%%%%%%%%%%%%%%%%%%%%%%%%%%%%%%%%%%%%%%%%%%%%%%%%%%%%%%%%%%%%%%%

\section{Concluding remarks and open problems}
\label{sec:conclusion}

Theorem~\ref{KPPR-variable} demonstrates that the threshold between polynomial and super-polynomial growth of the variable quasi-Ramsey numbers $R_{f}(k)$ occurs for $f(\ell)=\frac12\ell+\Theta(\sqrt{\ell\ln \ell})$.  Erd\H{o}s and Pach did not notice this phenomenon and indeed presumed a different outcome.
It is rare to see sharp asymptotic results in this area of mathematics, so this highlights the power of both graph discrepancy and the probabilistic method.

We may also ask for finer detail on the abrupt change in the variable quasi-Ramsey problem for minimum density around $\frac12$.
\begin{itemize}
\item
For $\eps>0$, what precisely is the least choice of $f(\ell)$ for which $R_{f}(k) = \Omega(k(\ln k)^{1+\eps})$? We only know it satisfies $\frac12\ell \leq f(\ell) \leq \frac12\ell+o(\sqrt{\ell\ln \ell})$.
\item
Does a form of Theorem~\ref{KPPR-variable}\ref{KPPR-variable,a} hold for $\nu = \nu(\ell) \to \infty$ as $\ell \to\infty$?
\end{itemize}

Our understanding of fixed quasi-Ramsey numbers $R^*_t(k)$ is less clear, even if thinning has brought us to a slightly better viewpoint.
We believe that it would be difficult to determine the second-order term in the polynomial-to-super-polynomial threshold for $R^*_t(k)$.  The threshold might be at $t = \frac12k+\Theta(\sqrt{k\ln k})$, this being the boundary case for super-polynomial behaviour in Proposition~\ref{prop:lower} or~\ref{prop:lower,spencer}.  We cannot rule out that the threshold is close to $t = \frac12k+\Theta(\ln k)$, this being the boundary case for polynomial behaviour in Theorem~\ref{EP-fixed}.
It is unlikely that one can use the thinning method to obtain sharp bounds for the fixed quasi-Ramsey number $R^*_t(k)$ for $t \geq \frac12(k-1)$. It seems that for this one would need bounds on the variable quasi-Ramsey numbers that contradict Theorem~\ref{KPPR-variable}\ref{KPPR-variable,b}.

We concentrated on the case of minimum density around $\frac12$, but it would also be interesting to better understand the parameters further away from the threshold. 
Intuitively, tightening the existing bounds in the exponential regime could be as difficult as the analogous problem for $R(k)$,  
but in the linear regime there is room for improvement, especially near the threshold.

Let us examine the bounds for $R^*_t(k)$. 
Fix $\alpha \in [0,1]$ and suppose $t = t(k)$ satisfies $t \sim \alpha(k-1)$ as $k \to \infty$.
If $\alpha > \frac12$, then Proposition~\ref{prop:lower} or~\ref{prop:lower,spencer} and the Erd\H{o}s--Szekeres-type bound of Chappell and Gimbel together give
\begin{align*}
\frac12\Lambda^*(1-\alpha)+o(1)
\le
\frac1k\ln R^*_t(k)
\le
2\alpha\ln 2+o(1).
\end{align*}
Recall that $\Lambda^*(1-\alpha) \downarrow 0$ as $\alpha \downarrow \frac12$ and $\Lambda^*(0)=\ln 2$.
It is curious that these bounds do not imply that $\frac1k\ln R^*_t(k)$ is strictly smaller than $\frac1k\ln R(k)$ for any $\alpha> \frac12$, but there might be a way to prove such a statement without improving the exponential bounds directly.
If $\frac14 \le \alpha < \frac12$, then the lower bound certified in Figure~\ref{fig:lowerex} and Theorem~\ref{thm:upper} (plugging in $\eps = 1/2-\alpha$) together give
\begin{align*}
2\alpha + 1 + o(1)
\le
\frac1kR^*_t(k)
\le
\sqrt{\left(\tfrac12-\alpha\right)^{-1}+1} + o(1).
\end{align*}
The thinning upper bound can be improved slightly, but close to $\alpha=\frac12$ a new idea may be needed for upper and lower bounds that agree up to a constant multiple, independent of $\frac12-\alpha$.
For $\alpha < \frac14$, there is the exact formula of Chappell and Gimbel.

To conclude, we reiterate a problem left open by Erd\H{o}s and Pach, which asks about arguably the most interesting case for $R^*_t(k)$, the symmetric choice $t = \frac12(k-1)$, rounded up or down. They showed that 
\begin{align*}
R^*_{\frac12(k-1)}(k) = \Omega\left(\frac{k\ln k}{\ln\ln k}\right)\quad\text{and}\quad R^*_{\frac12(k-1)}(k)=O(k^2),
\end{align*}
but what is the correct behaviour of $R^*_{\frac12(k-1)}(k)$? 
Note added: subsequent to the present work, three of the authors have improved the upper bound to $O(k\ln^2 k)$~\cite{KPR14+}.

%%%%%%%%%%%%%%%%%%%%%%%%%%%%%%%%%%%%%%%%%%%%%%%%%%%%%%%%%%%%%%%%%%%%%%%%%

\bibliographystyle{abbrv}
\bibliography{imramsey}

\end{document}